\documentclass[11pt]{amsart}
\usepackage{amssymb,amsmath,amsthm,amsfonts,mathrsfs}
\usepackage[hmargin=3cm,vmargin=3.5cm]{geometry}
\usepackage[dvipsnames,table,xcdraw]{xcolor}
\usepackage[colorlinks = true,
            linkcolor = Fuchsia,
            urlcolor  = ForestGreen,
            citecolor = WildStrawberry,
            anchorcolor = blue]{hyperref}

\usepackage{placeins}

\usepackage{stackengine}

\usepackage[all,2cell]{xy} \UseAllTwocells 
\usepackage{tikz-cd}

\usepackage{color}
\usepackage{graphicx,psfrag}
\usepackage{amscd}  
\usepackage{stmaryrd}  
\usepackage[all,2cell]{xy} \UseAllTwocells \SilentMatrices
\usepackage{tikz}
\usepackage[dvips]{epsfig}
\usepackage{comment}
\usepackage{kbordermatrix}
\usepackage{cancel,scalefnt}
\usepackage{blkarray}
\usepackage{multirow}

\theoremstyle{definition}
\newtheorem{theorem}{Theorem}[section]
\newtheorem{lemma}[theorem]{Lemma}
\newtheorem{remark}[theorem]{Remark}

\newtheorem{proposition}[theorem]{Proposition}

\newtheorem{corollary}[theorem]{Corollary}

\catcode`~=11 
\newcommand{\urltilde}{\kern -.15em\lower .7ex\hbox{~}\kern .04em}  
\catcode`~=13 

\usetikzlibrary{arrows,automata}
\usetikzlibrary{decorations.markings}
\usetikzlibrary{decorations.pathmorphing,shapes,decorations.text,shapes.geometric}
\usepackage{multirow}

\newcounter{sarrow}

\usepackage[colorinlistoftodos]{todonotes}

\title{Pairs of eventually constant maps and nilpotent pairs}

\author[Weixi Chen]{Weixi Chen}
\address{Department of Mathematics, Johns Hopkins University, Baltimore, MD 21218, USA}
\email{\href{mailto:wchen159@jh.edu}{wchen159@jh.edu}}

\author[Mee Seong Im]{Mee Seong Im}
\address{Department of Mathematics, Johns Hopkins University, Baltimore, MD 21218, USA}
\email{\href{mailto:meeseong@jhu.edu}{meeseong@jhu.edu}}

\author[Mikhail Khovanov]{Mikhail Khovanov}
\address{Department of Mathematics, Johns Hopkins University, Baltimore, MD 21218, USA}
\email{\href{mailto:khovanov@jhu.edu}{khovanov@jhu.edu}}

\author[Catherine Lillja]{Catherine Lillja}
\address{Department of Mathematics, Johns Hopkins University, Baltimore, MD 21218, USA}
\email{\href{mailto:clillja1@jh.edu}{clillja1@jh.edu}}

\author[Nicolas Rugo]{Nicolas Rugo}
\address{Department of Mathematics, Johns Hopkins University, Baltimore, MD 21218, USA}
\email{\href{mailto:nrugo1@jh.edu}{nrugo1@jh.edu}}

\makeatletter
\@namedef{subjclassname@2020}{%
  \textup{2020} Mathematics Subject Classification}

\date{December 3, 2025}

\providecommand{\keywords}[1]{\textbf{\textit{Key words and phrases.}} #1}


\begin{document}

\def\mfb{\mathfrak{b}}

\def\rank{\mathsf{rank}}

\def\E{\mathsf E}
\def\F{\mathbb{F}}
\def\I{\mathsf I}
\def\J{\mathsf J}
\def\R{\mathbb R}
\def\Q{\mathbb Q}
\def\Z{\mathbb Z}
\def\N{\mathbb N}
\def\C{\mathbb C}
\def\S{\mathbb S}
\def\Lin{\mathsf{Lin}}
\def\Nil{\mathsf{Nil}}
\def\SS{\mathbb S}
\def\GL{\mathsf{GL}}
\def\Graph{\mathsf{Graph}}

\def\for{\mathsf{for}}
\def\Hom{\mathsf{Hom}}
\def\End{\mathsf{End}}

\def\Der{\mathsf{Der}}
\def\Pol{\mathsf{Pol}}
\def\Span{\mathsf{Span}}

\newcommand{\dmod}{\mathsf{-mod}}
\newcommand{\comp}{\mathrm{comp}} 
\newcommand{\col}{\mathrm{col}}
\newcommand{\adm}{\mathrm{adm}}  
\newcommand{\Ob}{\mathrm{Ob}}
\newcommand{\Cob}{\mathsf{Cob}}
\newcommand{\UCob}{\mathsf{UCob}}
\newcommand{\COB}{\mathcal{COB}}
\newcommand{\ECob}{\mathsf{ECob}}
\newcommand{\id}{\mathsf{id}}
\newcommand{\undM}{\underline{M}}
\newcommand{\im}{\mathsf{im}}
\newcommand{\coker}{\mathsf{coker}}
\newcommand{\Aut}{\mathsf{Aut}}
\newcommand{\tripod}{\mathsf{Td}}
\newcommand{\BBC}{\mathbb{B}(\mathcal{C})}
\newcommand{\Pmod}{\mathrm{pmod}}
\newcommand{\gammaoneR}{\gamma_{1,R}}  
\newcommand{\gammaoneRbar}{\overline{\gamma}_{1,R}} 
\newcommand{\gammaoneRprime}
{\gamma'_{1,R}}
\newcommand{\gammaoneRbarprime}
{\overline{\gamma}'_{1,R}}
\newcommand{\qbinom}[3]{\genfrac{[}{]}{0pt}{}{#1}{#2}_{#3}}

\def\l{\lbrace}
\def\r{\rbrace}
\def\o{\otimes}
\def\lra{\longrightarrow}
\def\ed{\mathsf{ed}}
\def\Ext{\mathsf{Ext}}
\def\ker{\mathsf{ker}}
\def\mf{\mathfrak} 
\def\mcC{\mathcal{C}}
\def\bal{\mathsf{bal}}
\def\unbal{\mathsf{unbal}}
\def\mcN{\mathcal{N}}
\def\mcS{\mathcal{S}}  
\def\mcQC{\mathcal{QC}}
\def\mcA{\mathcal{A}}
\def\mcF{\mathcal{F}}
\def\mcE{\mathcal{E}}
\def\Fr{\mathsf{Fr}}  

\def\bbn{\mathbb{B}^n}
\def\ovb{\overline{b}}
\def\tr{{\sf tr}} 
\def\det{{\sf det }} 
\def\one{\mathbf{1}}   
\def\kk{\mathbf{k}}  
\def\gdim{\mathsf{gdim}}  
\def\rk{\mathsf{rk}}
\def\IET{\mathsf{IET}}
\def\SAF{\mathsf{SAF}}

\newcommand{\indexw}{\R_{>0}} 

\newcommand{\brak}[1]{\ensuremath{\left\langle #1\right\rangle}}
\newcommand{\oplusop}[1]{{\mathop{\oplus}\limits_{#1}}}
\newcommand{\addfigure}{\vspace{0.1in} \begin{center} {\color{red} ADD FIGURE} \end{center} \vspace{0.1in} }
\newcommand{\add}[1]{\vspace{0.1in} \begin{center} {\color{red} ADD FIGURE #1} \end{center} \vspace{0.1in} }
\newcommand{\vspin}{\vspace{0.1in} }

\newcommand\circled[1]{\tikz[baseline=(char.base)]{\node[shape=circle,draw,inner sep=1pt] (char) {${#1}$};}} 

\let\oldemptyset\emptyset
\let\emptyset\varnothing

\let\oldtocsection=\tocsection
\let\oldtocsubsection=\tocsubsection
\renewcommand{\tocsection}[2]{\hspace{0em}\oldtocsection{#1}{#2}}
\renewcommand{\tocsubsection}[2]{\hspace{1em}\oldtocsubsection{#1}{#2}}

\renewcommand{\kbldelim}{(}
\renewcommand{\kbrdelim}{)}

\def\MK#1{{\color{red}[MK: #1]}}
\def\bfred#1{{\color{red}#1}}

\def\MSI#1{{\color{ForestGreen}[MSI: #1]}}
\def\bfred#1{{\colo{magenta}#1}}


\begin{abstract}
Tom Leinster gave a bijective correspondence between the set of operators on a finite-dimensional vector space $V$ and the set of pairs consisting of a nilpotent operator and a vector in $V$. Over  a finite field this bijection implies that the probability that an operator be nilpotent is the reciprocal of the number of vectors in $V$. We  generalize this correspondence to pairs of operators between pairs of vector spaces and determine the probability that a random pair of operators be nilpotent. We also determine the set-theoretical counterpart of this construction and compute the number of eventually constant pairs of maps between two finite sets, closely related to the number of spanning trees in a complete bipartite graph.  
\end{abstract}

\maketitle
\tableofcontents

%
%

\section{Introduction}
\label{sec_intro} 
The nilpotent cone is foundational and prevalent in geometric and combinatorial representation theory. 
It provides key geometric and algebraic structures for the study of representations of Hecke algebras and related algebras~\cite{CG97}. 

Let $X$ be a finite-dimensional vector space over a field $\kk$ and $\mathcal{N}(X)\subset \End_{\kk}(X)$ be the set of nilpotent operators on $X$. Tom Leinster proved in \cite{Lei21} that there is a bijection between $\mathcal{N}(X)\times X$ and $\End_{\kk}(X)$. Working over a finite field with $q$ elements, the result specializes to an older theorem of Fine and Herstein \cite[Theorem 1]{FH58}, which says that the number of $n\times n$ nilpotent matrices is $q^{n(n-1)} = |X|^{n-1}$. Equivalently, the probability that a random operator be nilpotent is $q^{-n}$. A useful feature of Leinster's argument is that it requires very little calculation, and exposes a deeper structural pattern.
The set-theoretical analogue of nilpotent matrix counting reduces to the count of trees on $n$ labeled vertices, given by the celebrated Cayley's formula $n^{n-2}$; see~\cite{Lei21} and references there.

In \cite[Theorem 2]{FH58}, Fine and Herstein determine the number of nilpotent matrices over rings of the form $\mathbb{Z}/a\mathbb{Z}$, where $a$ is any positive integer. The mod $a$ count follows quite simply from the mod $p$ count, for $p$ a prime number, and this is the crux of the argument there.

\begin{figure}
    \centering
\begin{tikzpicture}[scale=0.5,decoration={
    markings,
    mark=at position 0.5 with {\arrow{>}}}]


\begin{scope}[shift={(0,0)}]


\draw[thick,fill] (0.2,2) arc (0:360:2mm);

\draw[thick,->] (0,2.45) arc (165:-165:1.65);

\end{scope}

\begin{scope}[shift={(10,0)}]


\draw[thick,fill] (0.2,2) arc (0:360:2mm);

\draw[thick,fill] (5.2,2) arc (0:360:2mm);

\node at (6.5,2) {$\Gamma$};

\draw[thick,->] (0.25,2.5) .. controls (1,4) and (4,4) .. (4.75,2.5);

\draw[thick,<-] (0.25,1.5) .. controls (1,0) and (4,0) .. (4.75,1.5);

\end{scope}


\end{tikzpicture}
    \caption{Left: a graph with a vertex and a loop. Right:  graph $\Gamma$ with 2 vertices and 2 arrows forming an oriented 2-cycle. }
    \label{fig_00010}
\end{figure}
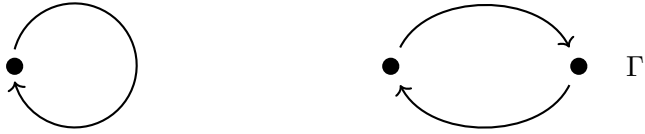
  
A linear operator on a vector space describes a representation of a quiver (oriented graph) with a single vertex and single loop at the vertex, shown in Figure~\ref{fig_00010} on the left. In the present paper we consider an extension of Leinster's construction to the oriented 2-cycle quiver $\Gamma$, shown in Figure~\ref{fig_00010} on the right. A finite-dimensional representation of $\Gamma$ over $\kk$ is given by a pair of vector spaces $V,W$ of dimensions $n,m$, respectively, and a pair of linear operators $f:V\lra W,g:W\lra V$. The analogue of the nilpotency condition for a linear operator in this case is the nilpotency of the composition $gf\in \End(V)$, equivalent to the nilpotency of $fg\in \End(W)$. 

Our main result is Theorem~\ref{thm_dim_nilpot_pairs}, which determines the number of nilpotent pairs $(f,g)$. Corollary~\ref{cor_prob_nil_pairs} rephrases it: the probability that a random pair of operators $(f,g)$ be nilpotent is $q^{-m} + q^{-n} - q^{-m-n}$. 

 The analogue of Leinster's bijection $\mathcal{N}(X)\times X\cong \End_{\kk}(X)$ is given earlier, by Theorem~\ref{thm_gen_nilpotent_two_vs}, which uses the notion of a \emph{balanced} vector, see Section~\ref{subsection_nilp_pairs_balanced_vec}. To get from that theorem to the count of nilpotent pairs in Theorem~\ref{thm_dim_nilpot_pairs} requires  looking at interactions between balanced and unbalanced vectors in $V$ and $W$, done in Section~\ref{subsection_nilp_pairs_bal_finite_field}. The key bijection  is provided by Lemma~\ref{lemma_quadruples}. The main result (Theorem~\ref{thm_dim_nilpot_pairs}) follows. Background material is given in Section~\ref{subsection_lin_alg}.

The set-theoretical counterpart of a nilpotent pair is given by a pair of finite sets and maps $f,g$ between them such that the composition $gf$ is eventually constant. In Section~\ref{subsection_background_graph}, we give a brief background on graph theory, and in Section~\ref{subsection_enumeration_eventually_const}, we enumerate eventually constant pairs of maps (Theorem~\ref{thm_eventually_const_pairs}).

The count of nilpotent matrices over the Boolean semiring is explained in Section~\ref{section_enumeration_Bool_semiring}.  
In Section~\ref{section_enumeration_Bool_semiring}, we prove that $n\times n$ nilpotent matrices over a Boolean semiring is enumerated by the number of directed acyclic graphs on $n$ ordered vertices (Proposition~\ref{prop_nilp_boolean_semiring}).

\section*{Acknowledgments}
The authors would like to thank Haihan Wu and Matthew Hamil for productive conversations. The authors would also like to thank the Department of Mathematics and Johns Hopkins University for the opportunity to conduct research in these last several months. The authors were partially supported by Simons Collaboration Award 994328. M.K. was also partially supported by NSF grant DMS-1807425.

\section{Eventually constant pairs of maps of sets}
\label{section_eventually_constant_graphs}

\subsection{Background on graph theory}
\label{subsection_background_graph}

We consider unoriented graphs without loops or multiple edges. 
 A simple path is a sequence of vertices in a graph, where each consecutive pair of vertices is connected by an edge, and no vertices, except possibly the endpoints, are repeated.
If there is a repetition, a simple path is called a cycle. A tree is an unoriented graph in which every pair of distinct vertices is connected by exactly one simple path; equivalently, a tree is a connected graph with no cycles.

A rooted tree is a tree together with a choice of vertex, which we call the root. A spanning tree inside a  connected graph $\Gamma$ is a subgraph $T \subseteq \Gamma$ such that $T$ is a tree, and maximal with respect to the property of being a tree inside $\Gamma$. This is equivalent to saying that $T$ is a tree that contains all vertices of $\Gamma$.

A graph is called bipartite or bicolorable if it is possible to color the vertices in two distinct colors, such as red and black, such that no two vertices of the same color are adjacent, i.e., connected by an edge. Recall that a graph is bipartite if and only if its cycles are all of even length~\cite[Theorem 5.3]{BM76} or~\cite{CCPS11}. Denote the complete bipartite graph by $K(m, n)$.

\subsection{The enumeration of eventually constant pairs of functions}
\label{subsection_enumeration_eventually_const}

We say that a map $f:X\lra X$ of a finite set $X$ is \emph{eventually constant} if $|\im(f^k)|=1$ for some $k\ge 0$, i.e., iterating $f$ results in a map that takes $X$ into a single element of $X$. 

Given sets $X,Y$ of cardinalities $m,n$, respectively, a  pair $(f,g)$ of maps $f:X\to Y$, $g:Y\to X$ is called an \emph{eventually constant} pair if the composite $gf:X\lra X$ is eventually constant. Equivalently, $fg:Y\lra Y$ is eventually constant. 
  Denote the set of eventually constant pairs by $P(m, n)$. The cardinality $|P(n,m)|$ is divisible by $mn$ since the eventually constant condition gives us unique ``final'' elements $x_0,y_0$ of $X,Y$, respectively, with  $f(x_0)=y_0$, $g(y_0)=x_0$ and $(gf)^k(x)=x_0, (fg)^k(y_0)=y_0$ for large $k$.  

Counting eventually constant pairs is a distributed version of Cayley's formula on counting trees. 

\begin{lemma}[Cayley's formula]
The number of unrooted trees with vertex set $X$ is $m^{m-2}$.
\end{lemma}

\begin{proof}
We refer the reader to~\cite[page 16]{A_Joyal81}, and see~\cite{Lei21} for the connection to eventually constant maps.
\end{proof}

\begin{lemma}\label{lemma_span}
    There are $m^{n-1}n^{m-1}$ spanning trees in the complete bipartite graph $K(m, n)$.
\end{lemma}

\begin{proof}
Originally proved in~\cite{FS58}, also see~\cite{AS90} and~\cite{Pak09,DK25}.
\end{proof}

The following theorem generalizes Cayley's formula: 
\begin{theorem}
\label{thm_eventually_const_pairs}
Let $X$ and $Y$ be finite sets of cardinality $m$ and $n$, respectively. Then there are $m^{n-1}n^{m-1}(m+n-1)$ eventually constant pairs of maps between $X$ and $Y$. 
\end{theorem}

\begin{proof}
 Pairs of maps $(f,g)$ between $X$ and $Y$ 
are in a bijection with bipartite oriented graphs on the vertex set pair $(X,Y)$, such that there is one oriented edge out of each vertex. The graph associated to $(f,g)$ has an oriented edge from $x\in X$ to $f(x)\in Y$ for each $x$ and an oriented edge from $y\in Y$ to $g(y)\in X$ for each $y$. Let $G(f,g)$ be the graph associated with $(f,g)$. 

An oriented cycle in a bipartite oriented graph has an even number of vertices, alternating between vertices in $X$ and $Y$. 
The graph $G(f,g)$ for an eventually constant pair $(f,g)$ has an (oriented) 2-cycle $x_0\lra f(x_0)=y_0 \lra g(y_0)=x_0$. It has no other oriented $2k$-cycles for any $k\ge 1$. 

Vice versa, if $G(f,g)$ has an oriented 2-cycle and no other oriented cycles then $(f,g)$ is an eventually constant pair. Consequently, there is a bijection between eventually constant pairs $(f,g)$ and oriented bipartite graphs $G$ on vertex set $(X,Y)$ with one edge out of each vertex such that $G$ has a unique 2-cycle and no other oriented cycles. 

Given a spanning tree $T$ in a complete (unoriented) graph $K(m,n)=K(X,Y)$, pick an edge $e$ of $T$. This edge connect vertices $x_0\in X$ and $y_0\in Y$. To the data $(T,e)$ associate an oriented bipartite graph $G(T,e)$ on vertex set $(X,Y)$ as follows. Deleting  edge $e$ from $T$ while keeping vertices $x_0,y_0$ gives a disjoint union of 2 trees $T_x,T_y$ with $x_0\in T_x,y_0\in T_y$ (so that $T\setminus \{e\}=T_x\sqcup T_y$). 

Orient each edge of $T_x$ towards $x_0$ and each edge of $T_y$ towards $y_0$. 
Define the graph $G(T,e)$ to contain all of these oriented edges and two oriented edges $x_0\lra y_0$ and $y_0\lra x_0$.
Thus, we're orienting all edges of $T$ except $e$ and converting $e$ to a pair of oriented edges between $x_0$ and $y_0$. 

Oriented graph $G(T,e)$ satisfies all properties above and determines a pair of eventually constant maps $(f,g)$. Vice versa, the graph $G(f,g)$ of an eventually constant pair determines a tree $T$ by taking the 2-cycle $x_0\lra y_0\lra x_0$ in $G(f,g)$, converting it to an unoriented edge $e$ of $T$ and adding all other edges of $G(f,g)$ without their orientations. We obtain the following result. 

\begin{lemma} The above correspondence is a bijection between eventually constant pairs of maps $(f,g)$ and spanning trees in $K(m,n)$ together with a choice of a edge in the tree. 
\end{lemma}

Any spanning tree in $K(m,n)$ contains $(m+n-1)$ edges, and there are $m^{n-1}n^{m-1}$ trees in $K(m,n)$, see Lemma~\ref{lemma_span}. This completes the proof of Theorem~\ref{thm_eventually_const_pairs}. 
\end{proof}

\begin{corollary}
\label{cor_prob_eventually_const}
    The probability that a pair $(f, g)$ of maps $f : X \to Y, \; g : Y \to X$ is eventually constant is $\frac{m+n-1}{mn}$.
\end{corollary}

\begin{proof}
There are $m^n n^m$ pairs of maps between $X$ and $Y$.
\end{proof}

\section{Pairs of nilpotent maps and a balanced vector}
\label{section_pairs_maps_nilpotent_finite_field}

Let $V,W$ be pairs of vector spaces of dimension $m$ and $n$, respectively, over any field $\kk$. Consider pairs of linear maps $(f, g)$, with $f \in \Hom(V,W)$ and $g \in \Hom(W,V)$. Denote the set (or cone) of nilpotent operators on $V$ by $\mathcal{N}(V) \subseteq \End(V)$. 

In this section, we generalize Leinster's main result~\cite[Theorem 5]{Lei21}.

\subsection{Linear algebra for pairs of vector spaces}
\label{subsection_lin_alg}

We begin by recalling some linear algebra. Firstly, a complement of a subspace $X$ in a vector space $V$ is a subspace $X'$ with $X \oplus X' = V$.

\begin{lemma}
\label{lemma_canonical_complement}
    There is a canonical bijection between $\Hom(V, W)$ and complements of $W$ in $V \oplus W$.
\end{lemma}

\begin{proof} A linear map $h:V\lra W$ determines the complement that consists of vectors $(v,h(v))\in V\oplus W$, $v\in V$. 
    (See ~\cite[Lemma 1]{Lei21}.) 
\end{proof}

\begin{lemma}
\label{Isomorphisms}
    Suppose $f : V \to W$ and $g : W \to V$ are linear maps. Then $gf$ and $fg$ are automorphisms of $V$ and $W$, respectively, if and only if both $f$ and $g$ are isomorphisms.
\end{lemma}

Lemmas~\ref{lemma_characterization_of_pairs} and~\ref{lemma_bijection_isomorphism} are generalizations of standard facts about linear operators to the case of pairs of maps. The first is a generalization of the Fitting decomposition.

\begin{lemma}
\label{lemma_characterization_of_pairs}
Let $V, W$ be finite-dimensional vector spaces. An ordered pair $(f,g)$ of linear maps, where $f\in \Hom(V,W)$ and $g\in \Hom(W,V)$,
is uniquely characterized by the following data. Take subspaces, which we write as $V_N, V_I$, such that $V = V_I \oplus V_N$, and subspaces $W_N, W_I$ of $W$ with $W = W_I \oplus W_N$. These are subject to the condition that $(gf)|_{V_I}$ and $(fg)|_{W_I}$ are automorphisms, and $(gf)|_{V_N} \in \mathcal{N}(V_N), \; (fg)|_{W_N} \in \mathcal{N}(W_N)$. 
\end{lemma}

\begin{proof}
We can take $V_I = \bigcap_{i \ge 0} \im((gf)^i)$, $V_N = \bigcup_{i \ge 0} \ker((gf)^i)$. 
Likewise, take $W_I = \bigcap_{i \ge 0} \im((fg)^i)$ and $W_N = \bigcup_{i \ge 0} \ker((fg)^i)$. The fact that $V = V_I \oplus V_N$ and $W = W_I \oplus W_N$ is the statement of Fitting's lemma, see ~\cite[pages 113--114]{Jac89}.  

It is immediate that, 
in the above subspace decompositions, $f$ and $g$ have the following block-diagonal form 
\[ f= \begin{pmatrix}
    S_1 & 0 \\
    0 & B_1
\end{pmatrix} ,  \quad\quad 
g = 
\begin{pmatrix}
    S_2 & 0 \\
    0 & B_2
\end{pmatrix}, 
\]
where $S_i$ are invertible, and $B_i$ compose to nilpotent operators.
Maps $S_1:=f|_{V_I} :V_I\to W_I$ and $S_2:=g|_{W_I} : W_I \to V_I$ are isomorphisms, while $(gf)|_{V_N}=B_2B_1$ and $(fg)|_{W_N}=B_1B_2$ are nilpotent operators. 
\end{proof}

The following statement is clear. 
\begin{lemma}
\label{lemma_bijection_isomorphism}
There is a natural bijection between ordered bases of finite-dimensional vector space $X$ of dimension $\ell$ and isomorphisms $X \stackrel{\simeq}{\lra} \kk^{\ell}$.
\end{lemma}

\begin{corollary}
\label{corr_on_ordered_bases}
There is a bijection between pairs of ordered bases of finite-dimensional vector spaces $V$ and $W$ with $\dim V=\dim W$, and pairs of isomorphisms $(f,g)\in \Hom(V,W)\times \Hom(W,V)$.
\end{corollary}

\subsection{Nilpotent pairs and balanced vectors}
\label{subsection_nilp_pairs_balanced_vec}
We continue to use the notations from Section~\ref{subsection_lin_alg}.
Consider $T=gf:V\lra V$ and $T'=fg:W\lra W$. We say the pair $(f,g)$ of linear maps is a \emph{nilpotent pair} if $T\in \mathcal{N}(V)$, equivalently, $T'\in \mathcal{N}(W)$. Denote the set of nilpotent pairs $(f,g)$ for $V,W$ by $\mathcal{N}(V,W)$. 

Let $v\in V$.
Form the subspace $T[v]$ of $V$, spanned by the linearly independent vectors $v,Tv,\dots, T^{a-1}(v)\not= 0$, with $T^{a}(v)=0$, and the subspace $T'[fv]$ of $W$ which is the span of linearly independent vectors $fv, T'fv,\ldots, (T')^{\ell-1}fv\not=0$, with $(T')^{\ell}fv=0$. 

There are two possible cases: 
\begin{itemize}
    \item We say that $v$ is \emph{balanced} if $a = \ell$, i.e., the spaces $T[v]$ and $T'[fv]$ have the same dimension, so that $f$ restricts to an isomorphism $T[v]\stackrel{f}{\lra} T'[fv]$. We call $a = \ell$ the \emph{length} of $v$. 
    \item The other case is $\ell=a-1$, so the second subspace has one dimension less than the first: $\dim T'[fv]=\dim T[v]-1$. We say that $v$ is \emph{unbalanced} in this case. 
\end{itemize}
Figure~\ref{fig_00009} shows the structure of images of $v$ under compositions of $f$ and $g$ when $v$ is balanced.
\begin{figure}
    \centering
\begin{tikzpicture}[scale=0.5,decoration={
    markings,
    mark=at position 0.5 with {\arrow{>}}}]


\begin{scope}[shift={(0,0)}]

\draw[thick] (-1,0) rectangle (5,6.5);

\node at (-2.5,5) {$V$};

\node at (-2.5,4.15) {$\dim m$};

\draw[thick,fill] (3.7,5.5) arc (0:360:2mm);

\node at (1.6,5.55) {$T^{\ell - 1}v$};

\node at (3.5,4) {$\vdots$};

\draw[thick,fill] (3.7,2.5) arc (0:360:2mm);

\node at (2.25,2.5) {$Tv$};

\draw[thick,fill] (3.7,1) arc (0:360:2mm);

\node at (2.5,1) {$v$};

\node at (6.4,7) {$0$};

\draw[thick,fill] (7.3,7) arc (0:360:2mm);

\node at (8.5,6.9) {$g$};

\draw[thick,<-] (7.5,7) -- (9,5.8);

\node at (6.0,6) {$f$};

\draw[thick,->] (4,5.5) -- (9,5.5);

\node at (7.5,4.90) {$g$};

\draw[thick,<-] (4,5) -- (9,4.3);

\node at (6.0,3) {$f$};

\draw[thick,->] (4,2.5) -- (9,2.5);

\draw[thick,->] (4,1) -- (9,1);

\node at (6.0,0.5) {$f$};

\node at (7.5,1.90) {$g$};

\draw[thick,<-] (4,2) -- (9,1.3);

\end{scope}


\begin{scope}[shift={(8,0)}]

\draw[thick] (0,0) rectangle (6,6.5);

\node at (7.5,5) {$W$};

\node at (7.5,4.15) {$\dim n$};

\draw[thick,fill] (1.7,5.5) arc (0:360:2mm);

\node at (3.6,5.6) {$(T')^{\ell - 1}w$};

\node at (1.5,4) {$\vdots$};

\draw[thick,fill] (1.7,2.5) arc (0:360:2mm);

\node at (2.9,2.5) {$T'w$};

\draw[thick,fill] (1.7,1) arc (0:360:2mm);

\node at (2.5,1) {$w$};

\end{scope}


\end{tikzpicture}
    \caption{Example of a balanced vector $v$ of length $\ell$ and its images under compositions of $f$ and $g$.}
    \label{fig_00009}
\end{figure}
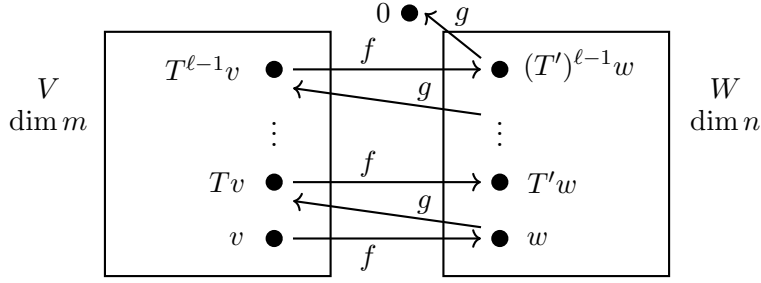

We have the following result, which is the analogue of~\cite[Theorem 5]{Lei21} for the quiver $\Gamma$ in Figure~\ref{fig_00010}:

\begin{theorem}
\label{thm_gen_nilpotent_two_vs}
There is a bijection 
\[
\{ (f, g, v): gf\in \mathcal{N}(V), v \mbox{ is balanced} \} 
\overset{\simeq}{\lra}
\Hom(V,W)\times \Hom(W,V).
\]
\end{theorem}

\begin{proof} At first, for each subspace $V_0\subset V$ pick a complementary subspace $V_0^{\perp}\subset V$ such that $V=V_0+V_0^{\perp}$ and $V_0\cap V_0^{\perp}=0$. Likewise, pick a complement $W_0^{\perp}$ for each subspace $W_0\subset W$.

Let $f, g$ be a nilpotent pair with a balanced vector $v$ of length $\ell$. We have subspaces $T[v]$ and $T'[fv]$ of the same dimension, each equipped with ordered bases $(v, Tv, \ldots, T^{\ell - 1}v)$ and $(fv, T'fv, \ldots, (T')^{\ell - 1} fv)$, respectively. In these bases, $f|_{T[v]}$ is the identity matrix and $g|_{T'[fv]}$ is a nilpotent Jordan block of size $\ell$, i.e.,
    $(J_{\ell})_{ij} = \delta_{i, j+1}, \text{ indices taken modulo }\ell$, by construction.

    We have complements $T[v]^{\perp}$ and $T'[fv]^{\perp}$ of $T[v]$ and $T'[fv]$ in $V$ and $W$, respectively, see the beginning of the proof. With respect to these subspaces, $f$ and $g$ are given by:
    \begin{equation}
    \label{eqn_f_g_nilp_pair}
    f = 
    \kbordermatrix{
                   &  T[v] & T[v]^{\perp}  \\ 
    T'[fv]         &  \I_{\ell}  & A \\
    T'[fv]^{\perp} &  0    & f|_{T[v]^{\perp}}
    }, 
    \quad\quad 
    g = \kbordermatrix{
                & T'[fv] & T'[fv]^{\perp} \\ 
    T[v]        & \J_{\ell}  & B \\
    T[v]^{\perp}&     0  & g|_{T'[fv]^{\perp}}
    },
    \end{equation}
    where $\I_{\ell}$ is the $\ell\times \ell$ identity matrix, $\J_{\ell}$ is the $\ell\times \ell$ Jordan block with $1$ immediately below the diagonal, $A$ is an $\ell\times (m-\ell)$ block, and $B$ is an $\ell\times (n-\ell)$ block.
    Since 
    \[ 
    gf = 
    \begin{pmatrix}
    C & D \\ 
    0 & E \\ 
    \end{pmatrix}
    \] where $C,D,E$ are blocks of appropriate size and $gf\in \mathcal{N}(V)$ if and only if $C\in\mathcal{N}(T[v])$ and $E\in\mathcal{N}(T[v]^{\perp})$, we see that 
    $(f|_{T[v]^{\perp}}, g|_{T'[fv]^{\perp}})$ is a nilpotent pair.
    Now, the data of subspaces with ordered bases is equivalent to a pair of isomorphisms $S_1 : T[v] \to T'[fv]$ and $S_2 : T'[fv] \to T[v]$ by Corollary~\ref{corr_on_ordered_bases}. 
    We replace the upper-left blocks of~\eqref{eqn_f_g_nilp_pair} with $S_1$ and $S_2$ since the original blocks gave the data of ordered bases of $T[v]$ and $T'[fv]$, and we are simply replacing them with their corresponding isomorphisms. 

There are many choices of complement for $T[v]$ and $T'[fv]$ (we fixed one choice). In order to resolve this, we can apply Lemma~\ref{lemma_canonical_complement}. The four relevant maps are \[
    A:T[v]^{\perp}\lra T'[fv], \ \ S_1:T[v]\lra T'[fv], \ \  S_2 : T'[fv] \to T[v], \ \ B:T'[fv]^{\perp}\lra T[v].
    \]
    We observe that $S_2 A \in \Hom(T[v]^\perp, T[v])$ corresponds exactly to a choice of complement for $T[v]$ in $V$. Likewise, $S_1 B \in \Hom(T'[fv]^\perp, T'[fv])$ corresponds exactly to a choice of complement for $T'[fv]$ in $W$. 
    Form new complementary subspaces of $T[v]$ and $T'[fv]$ using maps $S_2A$ and $S_1B$, respectively.  Define $f',g'$ to be given by the following block-diagonal matrices relative to these subspaces: 
    \[ 
    f' = \begin{pmatrix}
        S_1 & 0 \\
        0 & N_1
    \end{pmatrix}, \quad\quad 
    g' = \begin{pmatrix}
        S_2 & 0 \\
        0 & N_2
    \end{pmatrix},
    \]
    where 
    \[
    N_1 = f|_{T[v]^{\perp}}, \ \ 
    N_2 = g|_{T'[fv]^{\perp}}, 
    \]
    see \eqref{eqn_f_g_nilp_pair}. 
    This is also the Fitting decomposition of $(f',g')$.
     $N_i$ compose to nilpotent operators, and $S_i$ are isomorphisms. By Lemma~\ref{lemma_characterization_of_pairs}, we have obtained a unique pair of linear maps
    \[  (f', g') \in \Hom(V, W) \times \Hom(W, V). 
    \]
This concludes the isomorphism.
\end{proof}

\begin{corollary}\label{cor_averages_bal}
    The average number of balanced vectors in $V$ for a random nilpotent pair equals the average number of balanced vectors in $W$. 
\end{corollary}
This follows due to the symmetry between $V$ and $W$. Theorem~\ref{thm_gen_nilpotent_two_vs} holds with a balanced $v\in V$ replaced by a balanced $w\in W$. 

\subsection{Nilpotent pairs over a finite field}
\label{subsection_nilp_pairs_bal_finite_field}

Now, let $V$ and $W$ be $m$ and $n$-dimensional vector spaces, respectively, over the finite field $\mathbb{F}_q$, where $q$ is a prime power.
We would like to determine the number of nilpotent pairs $(f,g)$, i.e., $fg\in \mathcal{N}(W)$, given $m,n$. 
Define the $q$-binomial coefficient
\begin{equation}
\label{eqn_q_binom_coeff}
\qbinom{m}{r}{q} := \prod_{i=0}^{r-1}(q^m - q^i)\Big/\prod_{i=0}^{r-1}(q^r - q^i), 
\end{equation}
where if $r=0$, then $\prod_{i=0}^{r-1}(q^m - q^i):= 1$, and likewise for the denominator.
This is the number of $r$-dimensional subspaces of an $m$-dimensional $\mathbb{F}_q$-vector space.
\begin{lemma}
\label{lemma_no_rank_r_linear_maps}
    For $0\le r\le \min\{m,n\}$, the number of rank $r$ linear maps $f:V\rightarrow W$ is 
    \begin{equation}\label{eq_Nmnr}
    \mcN_{m,n;r} \ = \ \qbinom{m}{r}{q} \qbinom{n}{r}{q} \prod_{i=0}^{r-1}(q^r - q^i).
    \end{equation}
\end{lemma}
\begin{proof} 
The image of $f$ is an $r$-dimensional subspace of $W$, for which there are $\qbinom{n}{r}{q}$ choices. The number of surjective linear maps $V\lra \im(f)$ is the product of the number of codimension $r$ subspaces of $V$ (for $\ker f$; there are $\qbinom{m}{m-r}{q} = \qbinom{m}{r}{q}$ choices) and the size of $\GL(r,\F_q)$, given by the product term on the RHS of \eqref{eq_Nmnr}.  
\end{proof}

Denote by $\mathcal{N}(V,W)$ the set of nilpotent pairs $(f,g)$ as above ($f:V\lra W$ and $g:W\lra V$) and by $\mathcal{N}_{m,n}=|\mathcal{N}(V,W)|$ the cardinality of that set. 

\begin{theorem}
\label{thm_nilpotent_pairs_fin_field}
    The number of nilpotent pairs is given by the formula
    \begin{equation}
\label{eqn_nilpotent_pairs_Fq}
		\mathcal{N}_{m,n}=\sum_{r=0}^{\min\{m,n\}}   q^{mn-r} \mcN_{m,n;r}.
    \end{equation} 
\end{theorem}

\begin{proof}
    Given a rank $r$ linear map $f:V\rightarrow W$, we count the number of maps $g$ such that $fg$ is nilpotent.
    
    Without loss of generality, let $\{v_{r+1},\ldots,v_m\}$ be a basis for $\ker f\subset V$. Extend it to a basis $\{ v_1,\ldots, v_m\}$ of $V$. Since $f$ is linear, $\{f(v_1),\ldots,f(v_r)\}$ forms a basis of $\im f\subset W$. Extend it to a basis $\{w_1=f(v_1),\ldots, w_r=f(v_r),w_{r+1},\ldots, w_{n}\}$ of $W$.
        
        With respect to these two chosen bases for $V$ and $W$, the $n\times m$ matrix representation of $f$ is 
        \[ 
        F = \kbordermatrix{ 
            & r    & m-r \\ 
        r   & \I_r & 0 \\ 
        n-r & 0    & 0 
        },
        \]
        where $\I_r$ is the $r\times r$ identity matrix.
        Now let 
        \[
        G = 
        \kbordermatrix{
            & r & n-r \\ 
        r   & A & B \\ 
        m-r & C & D 
        }
        \] 
        be the matrix representation of $g$, where $A$ is $r \times r$, $B$ is $r \times (n-r)$, $C$ is $(m-r) \times r$, and $D$ is $(m-r) \times (n-r)$ matrices.
        The matrix representation for the composition $fg$ is the product:
        \[
        FG = \begin{pmatrix} \I_r & 0 \\ 0 & 0 \end{pmatrix} \begin{pmatrix} A & B \\ C & D \end{pmatrix} = \begin{pmatrix} A & B \\ 0 & 0 \end{pmatrix},
        \]
        which is a block upper-triangular matrix. Such a matrix is nilpotent if and only if its diagonal blocks are nilpotent. In our case, $A$ must be nilpotent. By ~\cite[Theorem 5]{Lei21}, the number of such $r\times r$ nilpotent matrices $A$ over $\mathbb{F}_q$ is $q^{r(r-1)}$. For blocks $B,C,D$, there are no restrictions. So the total number of such matrices $G$ (and thus maps $g$) for a fixed $f$ of rank $r$ is:
            $$ 
            q^{r(r-1)}  q^{r(n-r)}  q^{r(m-r)}  q^{(m-r)(n-r)} = q^{mn - r}. 
            $$
    Hence, for each $f$ of rank $r$, the same number of linear maps $g$ exists, such that $fg$ is nilpotent. 

    Now, by Lemma~\ref{lemma_no_rank_r_linear_maps}, we know the number of rank $r$ linear maps between $V$ and $W$. Summing over the rank of the linear map $f$, the result follows.
\end{proof}

\begin{proposition}
Equation~\eqref{eqn_nilpotent_pairs_Fq} can also be written as   
\[
\mcN_{m,n;r} = \sum_{T\in \Hom(V,W)} q^{mn-\rank(T)}.
\]
\end{proposition}

\begin{proof}
The number of rank $r$ linear maps $T:V\lra W$ is given by the expression in Lemma~\ref{lemma_no_rank_r_linear_maps}. 
In \eqref{eqn_nilpotent_pairs_Fq} each map is weighted with the coefficient $q^{mn - r}$.
\end{proof}

The following lemma is straightforward to establish. 
\begin{lemma}
    Let $(f,g)$ be a nilpotent pair. 
    \begin{enumerate}
    \item If $v\in V$, respectively $w\in W$, is balanced then every vector in the subspace $T[v]\subset V$, respectively in $T'[w]\subset W$, is balanced. 
    \item  If $v\in V$ is unbalanced then every vector in  $T[v]\setminus\{0\}$ is unbalanced. If $w\in W$ is unbalanced then every vector in $T'[w]\setminus \{0\}$ is unbalanced. 
    \item If $v\in V$ is balanced and $f(v)\not=0$ then $f(v)$ is unbalanced. If $w\in W$ is balanced and $g(w)\not=0$, then $g(w)$ is unbalanced. 
    \item If $v\in V$ is unbalanced, then $f(v)\in W$ is balanced. If $w\in W$ is unbalanced, then $g(w)\in V$ is balanced. 
    \end{enumerate}
\end{lemma}
In each of the three corollaries below, it is assumed that a nilpotent pair $(f,g)$ is given. 
\begin{corollary} \hfill 
\begin{enumerate}
\item 
    The set of balanced vectors for $(f,g)$ in $V$ is a union of $T$-stable linear subspaces of $V$. Likewise for balanced vectors in $W$. 
    \item 
    The set of unbalanced vectors for $(f,g)$ in $V$ union with the zero vector is a union of $T$-stable linear subspaces of $V$. Likewise for unbalanced vectors in $W$. 
    \end{enumerate}
\end{corollary}


\begin{corollary} Suppose $v\in V$ is balanced and $v'\in V$ is unbalanced. Then $T[v]\cap T[v']=0$. 
\end{corollary}

\begin{corollary} Suppose $v\in V$ and $w\in W$ are balanced. Then 
\[
T[v] \cap T[g(w)]=0\subset V, \ \ 
T'[f(v)]\cap T'[w]=0\subset W,
\]
and the sums $T[v] + T[g(w)]$ in $V$ and $T'[f(v)]+ T'[w]$ in $W$ are isomorphic to respective direct sums. 
\end{corollary} 

\begin{lemma}\label{lemma_quadruples} There is a bijection between quadruples  $(f,g,v_b,w_b)$ where $(f,g)$ is nilpotent, $v_b\not=0$ and $w_b$ are balanced in $V$ and $W$, respectively,  
and quadruples $(f,g,v_u,w_u)$ where $(f,g)$ is nilpotent and $v_u,w_u$ are unbalanced in $V$ and $W$, respectively.  
\end{lemma}

\begin{proof}
First, for each linear subspace $V_0\subset V$ pick a complementary subspace $V_0^{\perp}\subset V$, so that $V=V_0+ V_0^{\perp}$ and $V_0\cap V_0^{\perp} = 0$. 

Start with a quadruple $(f,g,v_b,w_b)$ as above with $v_b\not=0$ (call such a quadruple \emph{balanced}). Note the direct sum decomposition 
\[
T[v_b]\cong \kk v_b \oplus T(T[v_b]) = \kk v_b \oplus \kk\langle Tv_b, T^2v_b, \dots, T^{a-1}v_b\rangle, 
\]
where $a=\dim T[v_b]$.
Let $V'':=(T[v_b]+T[g w_b])^{\perp}$ and define 
$V':= V''+ T(T[v_b])+T[gw_b].$ Note that $\dim V'= \dim V - 1$, $V'\cap \kk v_b = 0$, and $V= \kk v_b + V'$. Also, with our notations, there is an equality of subspaces $T[Tv_b]=T(T[v_b])$. 

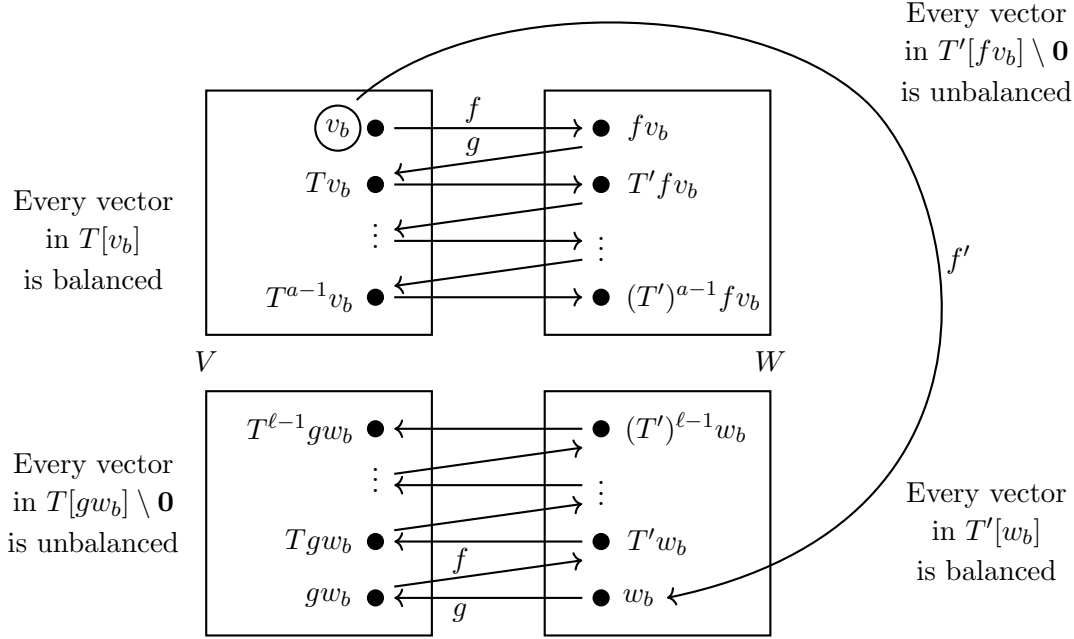
\begin{figure}
    \centering
\begin{tikzpicture}[scale=0.5,decoration={
    markings,
    mark=at position 0.5 with {\arrow{>}}}]


\begin{scope}[shift={(0,0)}]

\draw[thick] (-1,0) rectangle (5,6.5);

\node at (-1.0,-0.7) {$V$};

\node at (-4,3.5) {Every vector};

\node at (-4,2.5) {in $T[v_b]$};


\node at (-4,1.5) {is balanced};

\draw[thick,fill] (3.7,5.5) arc (0:360:2mm);

\draw[thick] (3.1,5.5) arc (0:360:0.6);

\draw[thick] (3,6.25) .. controls (6,9) and (13.5,9) .. (16.5,6.25);

\draw[thick,->] (16.5,6.25) .. controls (19,4) and (21,-5) .. (11.25,-7);

\node at (19,2) {$f'$};

\node at (2.5,5.5) {$v_b$};

\draw[thick,fill] (3.7,4) arc (0:360:2mm);

\node at (2.25,4) {$T v_b$};

\node at (3.5,2.85) {$\vdots$};

\draw[thick,fill] (3.7,1) arc (0:360:2mm);

\node at (1.75,1) {$T^{a-1}v_b$};

\node at (6.1,6) {$f$};

\draw[thick,->] (4,5.5) -- (9,5.5);

\node at (6.1,5) {$g$};

\draw[thick,<-] (4,4.3) -- (9,5);


\draw[thick,->] (4,4) -- (9,4);

\draw[thick,<-] (4,2.8) -- (9,3.5);


\draw[thick,->] (4,2.5) -- (9,2.5);


\draw[thick,<-] (4,1.3) -- (9,2);

\draw[thick,->] (4,1) -- (9,1);


\end{scope}


\begin{scope}[shift={(8,0)}]

\draw[thick] (0,0) rectangle (6,6.5);

\node at (6.0,-0.7) {$W$};

\node at (11.75,8.5) {Every vector};

\node at (11.75,7.5) {in $T'[fv_b]\setminus \mathbf{0}$};


\node at (11.75,6.5) {is unbalanced};

\draw[thick,fill] (1.7,5.5) arc (0:360:2mm);

\node at (2.75,5.5) {$f v_b$};

\draw[thick,fill] (1.7,4) arc (0:360:2mm);

\node at (3.15,4) {$T' f v_b$};

\node at (1.5,2.5) {$\vdots$};

\draw[thick,fill] (1.7,1) arc (0:360:2mm);

\node at (3.95,1) {$(T')^{a-1} fv_b$};

\end{scope}


\begin{scope}[shift={(0,-8)}]

\draw[thick] (-1,0) rectangle (5,6.5);

\node at (-4,4.5) {Every vector};

\node at (-4,3.5) {in $T[g w_b]\setminus \mathbf{0}$};


\node at (-4,2.5) {is unbalanced};

\draw[thick,fill] (3.7,5.5) arc (0:360:2mm);

\node at (1.5,5.5) {$T^{\ell-1}g w_b$};

\node at (3.5,4.35) {$\vdots$};

\node at (2.0,2.5) {$Tg w_b$};

\draw[thick,fill] (3.7,2.5) arc (0:360:2mm);

\draw[thick,fill] (3.7,1) arc (0:360:2mm);

\node at (2.25,1) {$g w_b$};


\draw[thick,<-] (4,5.5) -- (9,5.5);


\draw[thick,->] (4,4.3) -- (9,5);


\draw[thick,<-] (4,4) -- (9,4);

\draw[thick,->] (4,2.8) -- (9,3.5);


\draw[thick,<-] (4,2.5) -- (9,2.5);


\draw[thick,->] (4,1.3) -- (9,2);

\node at (5.75,2) {$f$};

\draw[thick,<-] (4,1) -- (9,1);

\node at (5.75,0.60) {$g$};

\end{scope}


\begin{scope}[shift={(8,-8)}]

\draw[thick] (0,0) rectangle (6,6.5);

\node at (11.75,3.75) {Every vector};

\node at (11.75,2.75) {in $T'[w_b]$};


\node at (11.75,1.75) {is balanced};

\draw[thick,fill] (1.7,5.5) arc (0:360:2mm);

\node at (3.75,5.5) {$(T')^{\ell-1} w_b$};

\node at (1.5,4) {$\vdots$};

\draw[thick,fill] (1.7,2.5) arc (0:360:2mm);

\node at (3,2.5) {$T' w_b$};

\draw[thick,fill] (1.7,1) arc (0:360:2mm);

\node at (2.5,1) {$w_b$};

\end{scope}


\end{tikzpicture}
    \caption{Top figure: $f$ is an isomorphism. Bottom figure: $g$ is an isomorphism. The map $f$ is redefined to $f'$ by changing it on $v_b$ to $f'(v_b)=w_b$. Then using the pair $(f',g)$, every vector in $T'[w_b]$, except $\mathbf{0}$, becomes unbalanced.  
    So we have the decompositions $V = T[v_b]\oplus T[gw_b] \oplus \widetilde{V}$ and $W = T'[w_b] \oplus T'[fv_b]\oplus \widetilde{W}$ for some vector spaces $\widetilde{V}$ and $\widetilde{W}$. 
    }
    \label{fig_20001}
\end{figure}

We define the map $f':V\lra W$ to coincide with $f$ on $V'$ and $f'(v_b)=w_b$. Thus, we are redefining $f$ on the vector $v_b$, see Figure~\ref{fig_20001}, but need its complement $V'$ for that definition. In the block decomposition of $V\cong T[v_b]\oplus T[gw_b]\oplus V''$ and $W\cong T'[fv_b]\oplus T'[w_b]\oplus W''$, where $W''$ is defined similarly,  maps $f$ and $f'$ have the form 

    \begin{eqnarray}
    & & \label{eqn_f_g_nilp_triple}
    f = 
    \kbordermatrix{
                   &  T[v_b] & T[g w_b] & V''  \\ 
    T'[fv_b]         &  \I_{a}  & 0 & \ast \\
    T'[w_b] &  0    & J_{\ell} & \ast  \\
W'' &  0    & 0 & f_{(1)} 
    }, 
    \quad\quad 
    g = \kbordermatrix{
                   &  T'[fv_b] & T'[ w_b] & W''  \\ 
    T[v_b]         &  J_{\ell}  & 0 & \ast \\
    T[gw_b] &  0    & \I_a & \ast  \\
V'' &  0    & 0 & g_{(1)} 
    },  \\
    & & 
    f' = 
    \kbordermatrix{
                   &  T_{\circ}[v_b] & T_\circ[T v_b] & V_\circ''  \\ 
    T'_{\circ}[w_b]         &  Z_{\ell}  & 0 & \ast \\
    T'_{\circ}[fv_b] &  0    & Z'_{a} & \ast  \\
W_\circ'' &  0    & 0 & f_{(1)} 
    },
    \end{eqnarray}
where $Z_\ell$ is an $\ell\times (\ell+1)$ matrix with 1's on the main diagonal and 0's everywhere else. $Z'_a$ is an $(a+1)\times a$ matrix with $1$'s right below the main diagonal and 0's everywhere else. (Note that the bases of $V,W$ used to write down matrices of $f,g$ above are different than the bases used for the matrix of $f'$.)

Here $T_{\circ}$ and $T'_{\circ}$ are the maps $T$ and $T'$ for the new pair $(f',g)$, and  $J_{\ell}$ is the standard matrix of a Jordan block of size $\ell$. 
The pair of entries  $(f_{(1)},g_{(1)})$ from lower right corners for $f$ and $g$ is a nilpotent pair. The subspace $W''$ is picked only to write down the block matrices in \eqref{eqn_f_g_nilp_triple}, while picking $V''$ is essential.

To a ``balanced'' quadruple $(f,g,v_b,w_b)$ the above construction assigns the ``unbalanced'' quadruple $(f',g,v_b,fv_b)$. Thus, 
\[
v_u := v_b, \ \ w_u:= f v_b. 
\]
Vectors $v_b\in V$ and $fv_b\in W$ are unbalanced for $(f',g)$, which requires $v_b\not=0$, as above. Denote this map by 
\begin{equation}
\rho: (f,g,v_b,w_b) \longmapsto (f',g,v_b,fv_b).
\end{equation}
It depends on the bijection $\perp$ above that picks a complement of each subspace and which is used to define $V''$ for each quadruple $(f,g,v_b,w_b)$. 

Next, define a map in the opposite direction, from unbalanced to balanced quadruples, as follows:  
\begin{equation}
\rho': (f',g,v_u,w_u) \longmapsto (f'',g,v_u,f'v_u).
\end{equation} 
Let $V_\circ''=(T_\circ[v_u]\oplus T_\circ[gw_u])^\perp$ and define $V_\circ'=V_\circ''\oplus T_\circ(T_{\circ}[v_u])\oplus T_\circ[gw_u]$.
We define the map $f'':V\to W$ to coincide with $f'$ on $V_\circ'$ and $f''(v_u)=w_u$. 

\vspace{0.07in}

We next show that $\rho$ and $\rho'$ are mutually-inverse bijections. 

(1) $\rho'\rho=\id$: map $\rho$ takes  $(f,g,v_b,w_b)$ to $(f',g,v_b,fv_b)$, while $\rho'$ takes $(f',g,v_b,fv_b)$ to $(f'',g,v_b,w_b)$ since $f'v_b=w_b$. It suffices to check $f''=f$. 
Maps $f''$ and $f'$ agree on $V_\circ'=V_\circ''\oplus T_\circ[T_\circ v_b]\oplus T_\circ[Tv_b]$, and $f$ and $f'$ agree on $V'= V''\oplus T[Tv_b]\oplus T[gw_b]$. Since $f''v_b=fv_b$ by definition of $f''$, the statement follows if we can show $V'=V_\circ'$.\\
Note $T_\circ v_b=gf'v_b=gw_b$, so $T_\circ[T_\circ v_b]=T_\circ[gw_b]$. Because $f'=f$ on $V'$, in particular $f'(gw_b)=f(gw_b)$, we know $T_\circ[T_\circ v_b]=T_\circ[gw_b]=T[gw_b]$. Also, $f$ and $f'$ agreeing on $V'$ implies $f'(Tv_b)=f(Tv_b)$, so $T_\circ[Tv_b]=T[Tv_b]$. Hence $T_\circ[T_\circ v_b]\oplus T_\circ[Tv_b]=T[Tv_b]\oplus T[gw_b]$. Adding $v_b$ to $T_\circ[T_\circ v_b]$ and $T[Tv_b]$ respectively, we have $V''=V''_\circ$ by the choice of them at the beginning of the proof. Hence $V'=V_\circ'$. 

\vspace{0.07in} 

(2) $\rho\rho'=\id$: map $\rho'$ takes $(f,g,v_u,w_u)$ to $(f',g,v_u,fv_u)$, while $\rho$ sends $(f',g,v_u,fv_u)$ to $(f'',g,v_u,w_u)$ since $f'v_u=w_u$. It suffices to check $f''=f$. 
Maps $f''$ and $f'$ agree on $V_\circ'=V_\circ''\oplus T_\circ[T_\circ v_u]\oplus T_\circ[T v_u]$, and $f$ and $f'$ agree on $V'=V''\oplus T[Tv_u]\oplus T[gw_u]$. Since $f''v_u=fv_u$ by definition of $f''$, the statement follows if we can show $V'=V_\circ'$.\\
Note $T_\circ v_u=gf'v_u=gw_u$, so $T_\circ[T_\circ v_u]=T_\circ[gw_u]$. Because $f'=f$ on $V'$, in particular $f'(gw_u)=f(gw_u)$, we know $T_\circ[T_\circ v_u]=T_\circ[gw_u]=T[gw_u]$. Also, $f$ and $f'$ agreeing on $V'$ implies $f'(Tv_u)=f(Tv_u)$, so $T_\circ[Tv_u]=T[Tv_u]$. Hence $T_\circ[T_\circ v_u]\oplus T_\circ[Tv_u]=T[Tv_u]\oplus T[gw_u]$. Adding $v_u$ to $T_\circ[T_\circ v_u]$ and $T[Tv_u]$ respectively, we have $V''=V''_\circ$ by the choice of them at the beginning of the proof. Hence $V'=V_\circ'$.  \\
Thus the bijection is established, completing the proof of the lemma.
\end{proof}


\begin{corollary}\label{cor_bij_quadruple}
    There is a bijection between the set of quadruples $(f,g,v_b,w_b)$ with $(f,g)$ nilpotent and $v_b,w_b$ balanced and the disjoint union of the following two sets: (1) the set of  quadruples $(f,g,v_u,w_u)$ with $(f,g)$ nilpotent and $v_u,w_u$ unbalanced, (2) the set of triples $(f,g,w_b)$ where $(f,g)$ is nilpotent and $w_b\in W$ is balanced. 
\end{corollary}

\begin{proof} Use the previous lemma and consider two cases: $v_b\not=0$ and $v_b=0$.
\end{proof}

\begin{theorem}
\label{thm_dim_nilpot_pairs}
There is a bijection
\begin{equation}\label{eq_biject}
\Hom(V,W) \times \Hom(W,V) \times (V\cup_0 W) \cong \mathcal{N}(V,W) \times V \times W. 
\end{equation}
The number of nilpotent pairs $(f,g)$ is 
\begin{equation} 
\label{eqn_num_nil_pairs_elegant_expression}
\mathcal{N}_{m,n}:=|\mcN(V,W)|=q^{2mn-m-n}(q^m + q^n - 1). 
\end{equation}
\end{theorem}

In \eqref{thm_dim_nilpot_pairs}, $V\cup_0 W$ denotes the union of $V$ and $W$ along their $0$ vectors. This union is naturally a subset of $V\oplus W$. 
Expression \eqref{eqn_num_nil_pairs_elegant_expression} is a closed formula, c.f.  \eqref{eqn_nilpotent_pairs_Fq}, see also Lemma~\ref{lemma_no_rank_r_linear_maps} for the notation. 
It would also be interesting to relate the count in  Theorem~\ref{thm_eventually_const_pairs} to that in \eqref{eqn_num_nil_pairs_elegant_expression}. 

\begin{proof}
The count of nilpotent pairs in \eqref{eqn_num_nil_pairs_elegant_expression} follows from the bijection in \eqref{eq_biject} by taking cardinalities of both sides. The cardinality of the LHS is $q^{2nm}(q^n+q^m-1)$ and that of the RHS is $\mcN_{n,m}q^nq^m$. 

To establish bijection \eqref{eq_biject} we 
introduce some notations:
\begin{itemize}
\item $\{f,g\}$ denotes the set of nilpotent pairs $(f,g)$,
\item $\{v\}$ denotes the set of vectors $v\in V$, 
    \item $\{f,g,v\}$ denotes the set of triples consisting of nilpotent pair and a vector $v\in V$, 
    \item $\{f,g,v_b\}$, respectively $\{f,g,v_u\}$ denotes the set of triples as above together with a choice of a balanced vector $v_b\in V$, respectively unbalanced vector $v_u\in V$, 
    \item $\{f,g,v_b,w_u\}$ denotes the set of quadruples consisting of a nilpotent pair, a balanced $v_b\in V$, and an unbalanced $w_u\in W$, 
    \item $[f,g]$ denotes the set of all pairs of maps $f\in \Hom(V,W), g\in \Hom(W,V)$.  
\end{itemize} 
We omit a full list of such  notations. 

\vspace{0.07in} 

From Theorem~\ref{thm_gen_nilpotent_two_vs} we have bijections 
\begin{equation}\label{eq_bij_1}
\{f,g,v_b\} \sim [f,g]\sim \{f,g,w_b\}. 
\end{equation}
Corollary~\ref{cor_bij_quadruple} gives a bijection 
\begin{equation}\label{eq_bij_2}
\{f,g,v_b,w_b\} \sim  \{f,g,v_u,w_u\}\sqcup [f,g]. 
\end{equation}

Let us separate $v\in V$ into balanced $v_b$ and unbalanced $v_u$ and likewise for $w's$, giving bijections
\begin{equation}\label{eq_bij_3}
    \{f,g,v\}\sim\{f,g,v_b\} \sqcup \{f,g,v_u\}\sim [f,g]\sqcup \{f,g,v_u\},
    \end{equation}
and likewise for $w$'s. We have 
\begin{equation}\label{eq_bij_4}
    \{f,g,v_b,w\}\sim [f,g]\times \{w\}\sim [f,g]\times W, \ \ \{f,g,v,w_b\}\sim [f,g]\times \{v\}\sim [f,g]\times V. 
\end{equation}
Furthermore, 
\begin{eqnarray*}
    \{f,g,v,w\}\sqcup[f,g] & \sim & \{f,g,v_b,w_b\}\sqcup \{f,g,v_b,w_u\}\sqcup\{f,g,v_u,w_b\}\sqcup(\{f,g,v_u,w_u\} \sqcup [f,g]) \\
    & \sim & 
    \left(\{f,g,v_b,w_b\}\sqcup \{f,g,v_b,w_u\}\right) \sqcup 
    \left( \{f,g,v_u,w_b\}\sqcup \{f,g,v_b,w_b\} 
    \right) \\
    & \sim &
    \{f,g,v_b\}\times W \sqcup \{f,g,w_b\}\times V  \\
    & \sim & [f,g] \times (V\sqcup W). 
\end{eqnarray*}
In the second step bijection~\eqref{eq_bij_2} is used. The above chain of bijections gives a bijection
\begin{equation}\label{eq_bij_5}
\{f,g,v,w\}\sqcup[f,g] \sim [f,g] \times (V\sqcup W). 
\end{equation}
Modifying this bijection to remove a copy of the set $[f,g]$ from both sides gives a bijection 
\begin{equation}\label{eq_bij_6}
\{f,g,v,w\}\sim [f,g] \times (V\cup_0 W)
\end{equation}
for \eqref{eq_biject}.
\end{proof}

\begin{corollary}
\label{cor_prob_nil_pairs}
The probability that a random pair $(f,g)$ be nilpotent is 
\[ 
\dfrac{q^m + q^n - 1}{q^{m+n}}
 = q^{-m} + q^{-n} - q^{-m-n}.
\]
\end{corollary}

\begin{proof}
Follows at once from Theorem~\ref{thm_dim_nilpot_pairs}. 
\end{proof}

If $m$ is fixed and dimension $n$ of $W$ goes to infinity, this probability approaches 
$q^{-m} = 1/|V|$, 
    which is the probability that a random endomorphism of $V$ be nilpotent~\cite{Lei21}.

\begin{remark} Theorem~\ref{thm_gen_nilpotent_two_vs} and Corollary~\ref{cor_prob_nil_pairs} allow to compute the probability that a randomly chosen $v\in V$ be balanced for a random nilpotent pair $(V,W)$: 
\[
\mathsf{prob}(v \ \mathrm{is\ balanced}) 
= \frac{|\{f,g,v_b\}|}{|\mathcal{N}(V,W)|\times |V|} 
= \frac{q^n}{q^m + q^n - 1}.
\]
In particular, if $\dim V = m > n = \dim W$, this probability is less that $1/2$ and given by 
\[
\mathsf{prob}(v \ \mathrm{is\ balanced})
= \frac{1}{q^{m-n} + 1 - q^{-n}}. 
\]
For $m \gg n$, 
this probability is close to 
$q^{n-m}$. 
If $m<n$, 
    the probability is greater than $1/2$ and 
\[
\mathsf{prob}(v \ \mathrm{is\ balanced}) 
  = 
\frac{1}{1+q^{-n}(q^{m}-1)} \approx 1 - q^{m-n},
\]
where $\approx$ is for $n-m, n\gg 0$. 
If $m = n$,  
\[
\mathsf{prob}(v \ \mathrm{is\ balanced})= \frac{1}{2- q^{-m}}; 
\]
 slightly greater than $1/2$ but exponentially close to it for $m = n \gg 0$.  
\end{remark}

\begin{remark} A nilpotent endomorphism of a finite-dimensional vector space $V$ over $\kk$ is isomorphic to a direct sum of nilpotent Jordan blocks $J_r$ of sizes $r\ge 1$. The latter classify isomorphism classes of indecomposable nilpotent representations of the quiver with one vertex and one loop. See Figure~\ref{fig_00010} left. For the quiver $\Gamma$ in Figure~\ref{fig_00010} right, indecomposable nilpotent representations $(V,W,f,g)$ can be classified using balanced and unbalanced vectors.  Namely, 
\begin{itemize}
\item 
A balanced $v_b\in V$ with $\dim T[v_b]$ $=$ $\ell$ generates an indecomposable module $(T[v_b]$, $T'[fv_b])$ of dimension $(\ell,\ell )$. 
\item An unbalanced $v_u\in V$ with $\dim T[v_u]=\ell$ generates an indecomposable module $(T[v_u]$, $T'[fv_u])$ of dimension $(\ell, \ell-1)$. 
\item A balanced $w_b\in W$ with $\dim T'[w_b]$ $=$ $\ell$ generates an indecomposable module $(T[gw_b]$, $T'[w_b])$ of dimension $(\ell,\ell )$.
\item An unbalanced $w_u\in W$ with $\dim T'[w_b]$ $=$ $\ell$ generates an indecomposable module $(T[gw_u]$, $T'[w_u])$ of dimension $(\ell-1, \ell)$.
\end{itemize}
The dimension of a representation is understood to be the pair $(\dim V,\dim W)$.
Possible dimensions of nilpotent indecomposables are $(m, n)$ with $|m-n| \le 1$. Note that there are two non-isomorphic indecomposables of dimension $(m,m)$, 
    $m\ge 1$, 
        and one indecomposable when $n = m \pm 1$. 
\end{remark} 

\begin{remark} 
Theorem~\ref{thm_dim_nilpot_pairs} together with a simple count of maps implies that the number of triples: a nilpotent pair $(f,g)$ and a balanced vector $v_b\in V$ of length $\ell$, i.e., a balanced vector such that $\dim T[v_b]=\ell$, is given by 
 \[
\begin{split}
q^{2mn+\ell - (m + n)(\ell + 1)}(q^{m}+q^{n}-q^{\ell})
\prod_{i=0}^{\ell-1} (q^m - q^i) \prod_{j=0}^{\ell-1} (q^n - q^j) .
\end{split}
\]
\end{remark}

\begin{remark} It would be natural to give a motivic lifting of Theorem~\ref{thm_dim_nilpot_pairs}, similar to the motivic interpretation of the count of nilpotent endomorphisms in~\cite[Example 3.6.(2)]{GR25}.
\end{remark}



\section{Nilpotent endomorphisms over the Boolean semiring}
\label{section_enumeration_Bool_semiring}

The notion of a nilpotent operator makes sense for 
linear operators on semimodules over semirings. Consider the Boolean semiring
 $\mathbb{B} = \{ 0,1: 1+1 = 1\}$. The number of idempotents in the matrix semiring $M_n(\mathbb{B})$ was obtained by Butler~\cite{Butler1972}. We give a count of nilpotent matrices in $M_n(\mathbb{B})$, denoting the set of such matrices by $\mcN_n(\mathbb{B})$.

\begin{lemma}
Let $A = (A_{ij})\in \mcN_n(\mathbb{B})$. Then 
$A_{ii}=0$ for all $i$, and if $A_{ij}=1$, then $A_{ji}=0$. A Boolean matrix is nilpotent iff there is no oriented cycle of $1$'s among its entries, i.e., there is no sequence $i_1,i_2,\ldots, i_k$ such that $A_{i_1,i_2}=A_{i_2,i_3}=\ldots =A_{i_k,i_1}=1$.  
\end{lemma}

\begin{proof}
A $1$ on a diagonal in $A$ remains $1$ in any power of $A$. An oriented $k$-cycle of $1$'s create $1$ on the diagonal in $A^k$. If $A$ has no oriented cycles of $1$'s (including 1-cycles, i.e., diagonal entries $1$) then $A^n=0$ necessarily. 
\end{proof}

\begin{corollary}
\label{prop_nilp_boolean_semiring}
The set $\mcN_n(\mathbb{B})$  is in a bijection with the set of directed acyclic graphs (DAGs) on $n$ ordered vertices.
\end{corollary}

\begin{proof}
    Indeed, to a direct acyclic graph $G$ on vertices $1,2,\dots,n$ assign the matrix $A$ with $A_{ij}=1$ iff there is an oriented edge from vertex $j$ to vertex $i$, otherwise $A_{ij}=0$. This is a bijection between DAGs and nilpotent Boolean matrices in view of the above lemma. 
\end{proof}

\begin{corollary}
\label{lemma_Boolean_nilpotent}
Let $a_n=|\mcN_n(\mathbb{B})|$ be the number of nilpotent Boolean matrices of size $n$. These numbers satisfy the recurrence relation 
\begin{equation}
\label{eqn_recurrence_reln_nilpot_Bool}
a_0=1, \ \ \   
a_n = \displaystyle{\sum_{k=1}^{n}} (-1)^{k-1} \binom{n}{k} 2^{k(n-k)}a_{n-k}  
\quad
\mbox{ for } 
n\ge 1. 
\end{equation} 
\end{corollary}

\begin{proof}
See \cite{oeisA003024} for the references that the number of DAGs on $n$ labeled vertices satisfies relation \eqref{eqn_recurrence_reln_nilpot_Bool}.
\end{proof}

Matrices in $M_n(\mathbb{B})$ are in a bijection with endomorphisms of the free $\mathbb{B}$-semimodule $\mathbb{B}^n$, and the count of nilpotent matrices is given above. One can ask a natural question about counting nilpotent endomorphisms of a general finitely-generated projective semimodule $P$ instead of the free semimodule $\mathbb{B}^n$. Isomorphism classes of such semimodules are in a bijection with isomorphism classes of finite topological spaces $X$, see~\cite[Section 3.2]{IK-top-automata}, for instance. Semimodule $P$ associated to $X$ as above has its elements the open subsets of $X$. Zero element of $P$ is the empty subset of $X$ and the sum in $P$ is given by the union of sets. Counting nilpotent endomorphisms of $P$ is a combinatorial problem that generalizes the count of labeled DAGs, and it may be interesting for at least some classes of finite topological spaces. The linear counterpart of this problem might be counting nilpotent endomorphisms of a finite-dimensional module $M$ over a finite-dimensional $\mathbb{F}_q$-algebra. 




\FloatBarrier

\bibliographystyle{amsalpha} 
\bibliography{nilpotent_finite}

\end{document}